\documentclass[12pt,twoside]{amsart}
\usepackage{amsmath, amsthm, amscd, amsfonts, amssymb, graphicx}
\usepackage{enumerate}
\usepackage[colorlinks=true,
linkcolor=blue,
urlcolor=cyan,
citecolor=red]{hyperref}
\usepackage{mathrsfs}
\newtheorem{theorem}{Theorem}[section]
\newtheorem{lemma}{Lemma}[section]
\newtheorem{remark}{Remark}[section]

\newtheorem{corollary}{Corollary}[section]

\numberwithin{equation}{section}

\begin{document}
	
\title{Further numerical radius inequalities}
\author{Mohammad Sababheh, Cristian Conde, and Hamid Reza Moradi}

\begin{abstract}
In this article, we present some new inequalities for the numerical radius of products of Hilbert space operators and the generalized Aluthge transform. In particular, we show some upper bounds for $\omega(ABC+DEF)$ using the celebrated Buzano inequality, then some consequences that generalize some results from the literature are discussed. After that, inequalities that involve the generalized Aluthge transform are shown using some known bounds for the numerical radius of the product of two operators.
\end{abstract}

\maketitle
%------------------------------------------------------------------------------------%
\pagestyle{myheadings}
\markboth{\centerline {}}
{\centerline {}}
\bigskip
\bigskip
%------------------------------------------------------------------------------------%
%------------------------------------------------------------------------------------% 
\section{Introduction}
Given a complex Hilbert space $\mathcal{H}$, endowed with the inner product $\left<\cdot,\cdot\right>$, we use $\mathcal{B}(\mathcal{H})$ to denote the $C^*$-algebra of all bounded linear operators on $\mathcal{H}$. For $A\in\mathcal{B}(\mathcal{H}),$ the numerical radius and the operator norm of $A$ are defined respectively by
$$\omega(A)=\sup_{\|x\|=1}\left|\left<Ax,x\right>\right|\;{\text{and}}\;\|A\|=\sup_{\|x\|=1}\|Ax\|.$$ It is well known that if $A$ is normal, in the sense that $A^*A=AA^*$, then $\|A\|=\omega(A)$. However, for non-normal operators, this equality fails. In general, the following holds for any $A\in\mathcal{B}(\mathcal{H})$:
\begin{equation}\label{eq_equivalent}
\frac{1}{2}\|A\|\leq \omega(A)\leq \|A\|.
\end{equation} 
The importance of this inequality is the way it approximates $\omega(A)$ in terms of $\|A\|$; a more manageable quantity to compute than $\omega(A)$. Sharpening \eqref{eq_equivalent} and other inequalities for the numerical radius has been the interest of numerous researchers in the past few years; see \cite {drag_1,E3,sab_czech,moradi1, sab_laa, sab_lama, BBP,1,sh}
 for example.

Among the most interesting bounds, we have the following lemma.
\begin{lemma}\label{lem_kitt}\cite{Ki05,Ki2}
Let $A,B\in\mathcal{B}(\mathcal{H})$. Then 
\begin{equation}\label{10}
\omega \left( AB \right)\le \frac{1}{2}\left\| {{\left| {{A}^{*}} \right|}^{2}}+{{\left| B \right|}^{2}} \right\|,
\end{equation}
$$\omega^2(A)\leq\frac{1}{2}\|\;|A|^2+|A^*|^2\;\|,$$
$$\omega(A)\leq\frac{1}{2}\|\;|A|+|A^*|\;\|,$$
and
\begin{equation}\label{16}
\omega(A)\leq\frac{1}{2}\left(\|A\|+\|A^2\|^{\frac{1}{2}}\right).
\end{equation}
\end{lemma}
In studying the notion of the numerical radius, the Aluthge transform has played an interesting role. Recall that this transform is defined for $A\in\mathcal{B}(\mathcal{H})$ by $\widetilde{A}=|A|^{\frac{1}{2}}U|A|^{\frac{1}{2}}$ where $U$ is the partial isometry that appears in the polar decomposition $A=U|A|$ of $A$. In fact, for $0\leq t\leq 1,$ the generalized Aluthge transform of $A$ is defined by $\widetilde{A_t}=|A|^{1-t}U|A|^t.$ In \cite{4}, it is shown that
\begin{equation}\label{eq_yam_1}
\omega(A)\leq \frac{1}{2}\left(\|A\|+\omega(\widetilde{A})\right).
\end{equation}
Using the inequalities \cite{4}
\begin{equation}\label{eq_yam_2}
\omega(\widetilde{A})\leq \|\widetilde{A}\|\leq \|A^2\|^{\frac{1}{2}}\leq \|A\|,
\end{equation}
it follows that \eqref{eq_yam_1} refines the second inequality in \eqref{eq_equivalent} and the inequality \eqref{16}.

Although the numerical radius and the operator norm are equivalent, the numerical radius differs from the operator norm being not submultiplicative. That is, if $A,B\in\mathcal{B}(\mathcal{H})$, then $\|AB\|\leq\|A\|\;\|B\|,$ but $\omega(AB)\not\leq \omega(A)\omega(B).$ In fact, it has been a standing question about the best possible constant in $\omega(AB)\leq k\omega(A)\omega(B).$ We refer the reader to \cite{aak} for discussion of submultilicativity of $\omega(\cdot).$

It was shown in \cite{9} that if $A,B\in \mathcal B\left( \mathcal H \right)$, then
\begin{equation*}
\omega \left( AB+ BA \right)\le 2\sqrt{2}\omega \left( A \right)\left\| B \right\|
\end{equation*}
and
\begin{equation}\label{12}
\omega(AB+B^*A)\leq 2\omega(A)\|B\|.
\end{equation}

We refer the reader to \cite{5} for different results regarding the numerical radius of product and commutators of Hilbert space operators.

Some critical tools in obtaining numerical radius inequalities are listed next.

\begin{lemma}\label{11}
\cite{7} Let $A,B\in \mathcal B\left( \mathcal H \right)$ be two positive operators. Then
\[\left\| A+B \right\|\le \max \left\{ \left\| A \right\|,\left\| B \right\| \right\}+\left\| {{\left| A \right|}^{\frac{1}{2}}}{{\left| B \right|}^{\frac{1}{2}}} \right\|.\]
\end{lemma}

\begin{lemma}\label{14}\cite{kato}
Let $T\in \mathcal B\left( \mathcal H \right)$ and $0\le t\le 1$. Then for any $x,y\in \mathcal H$,
\[\left| \left\langle Tx,y \right\rangle  \right|\le \sqrt{\left\langle {{\left| T \right|}^{2t}}x,x \right\rangle \left\langle {{\left| {{T}^{*}} \right|}^{2\left( 1-t \right)}}y,y \right\rangle}.\]
\end{lemma}
\begin{lemma}\label{7}\cite{bu}
Let $x,y,e\in \mathcal H$ with $\left\| e \right\|=1$. Then 
	\[\left| \left\langle x,e \right\rangle \left\langle y,e \right\rangle  \right|\le \frac{1}{2}\left( \left| \left\langle x,y \right\rangle  \right|+\left\| x \right\|\left\| y \right\| \right).\]
\end{lemma}

%\begin{lemma}\label{AB} Let $A,B\in \mathcal B\left( \mathcal H \right)$ such that $AB$ is self-adjoint,  then $\|AB\|\leq \|\mathfrak R(BA)\|.$\end{lemma}

%\begin{lemma}\label{7}
%\cite[Theorem 2.2]{3} Let $A,B,C,D\in \mathcal B\left( \mathcal H \right)$. Then
%\[r\left( AB+CD \right)\le \frac{1}{2}\left( \omega \left( BA \right)+\omega \left( DC \right)+\sqrt{{{\left( \omega \left( BA \right)-\omega \left( DC \right) \right)}^{2}}+4\left\| BC \right\|\left\| DA \right\|} \right).\]
%\end{lemma}

\begin{lemma}\label{6}
\cite[Theorem 3.4]{aak} Let $A,B\in \mathcal B\left( \mathcal H \right)$. Then
\[\omega \left( A+B \right)\le \frac{1}{2}\left( \left( \omega \left( A \right)+\omega \left( B \right) \right)+\sqrt{{{\left( \omega \left( A \right)-\omega \left( B \right) \right)}^{2}}+4\underset{\theta \in \mathbb{R}}{\mathop{\sup }}\,\left\| \mathfrak R\left( {{e}^{\textup i\theta }}A \right)\mathfrak R\left( {{e}^{\textup i\theta }}B \right) \right\|} \right).\]
\end{lemma}

This inequality is clearly a refinement of the triangle inequality $\omega(A+B)\leq \omega(A)+\omega(B)$, due to the fact that \cite{4}
\begin{equation*}
\omega(A)=\sup_{\theta \in \mathbb{R}}\|\mathfrak{R}(e^{i\theta}A)\|.
\end{equation*}

\begin{lemma}\label{lem_block}\cite[Lemma 2.1]{2}
Let $X,Y\in\mathcal{B}(\mathcal{H})$. Then
\[\omega \left( \begin{matrix}
   O & X  \\
   Y & O  \\
\end{matrix} \right)=\omega \left( \begin{matrix}
   O & Y  \\
   X & O  \\
\end{matrix} \right)=\omega \left( \begin{matrix}
   O & X  \\
   -Y & O  \\
\end{matrix} \right).\]
\end{lemma}

Our target in this article is to show some inequalities for the numerical radius of products of Hilbert space operators, such as finding bounds for $\omega(ABC+DEF)$ and some inequalities for the Aluthge transform that refine some of the known results in the literature. Many consequences and other inequalities for $\omega(AB), \omega(A), \omega(A+B)$ and $\omega(\widetilde{T}_t)$ will be presented. In this context, $\widetilde{T}_t$ denotes the generalized Aluthge transform, defined for $T\in\mathcal{B}(\mathcal{H})$ and $0\leq t\leq 1$ by
$$\widetilde{T}_{t}=|T|^{1-t}U|T|^t,$$ where $U$ is the partial isometry in the polar decomposition $T=U|T|$ of $T$.

\section{Main Result}
In this section, we present our main results. The first part of the section will deal mainly with inequalities of the numerical radius of the product of operators; then, we will discuss some inequalities that involve the Aluthge transform.

Block operators have played a significant role in obtaining numerical radius inequalities for certain Hilbert space operators. In the following, we present the following excellent application of Lemma \ref{6}, where more straightforward block operators can be used as a bound for $\omega(AB+CD)$.
\begin{theorem}\label{8}
Let $A,B,C,D\in \mathcal B\left( \mathcal H \right)$. Then
\[\omega \left( AB+CD \right)\le \omega \left( \left[ \begin{matrix}
   O & AB  \\
   CD & O  \\
\end{matrix} \right] \right)+\frac{1}{2}\sqrt{ \left\|{{\left( AB \right)}^{2}}\right\| +\left\| {{\left( {{D}^{*}}{{C}^{*}} \right)}^{2}}\right\| +\left\| AB{{D}^{*}}{{C}^{*}}+{{D}^{*}}{{C}^{*}}AB \right\|}.\]
In particular,
\[\omega \left( A\pm C \right)\le \omega \left( \left[ \begin{matrix}
   O & A  \\
   C & O  \\
\end{matrix} \right] \right)+\frac{1}{2}\sqrt{\left\| {{A}^{2}} \right\|+\left\| {{C}^{2}}\right\|+\left\| A{{C}^{*}}+{{C}^{*}}A \right\|}.\]
\end{theorem}
\begin{proof}
Applying  Lemma \ref{6} in what follows, we notice that
{\small		
\[\begin{aligned}
  & \omega \left( AB+CD \right) \\ 
 & =\omega \left( \left[ \begin{matrix}
   O & AB+CD  \\
   AB+CD & O  \\
\end{matrix} \right] \right) \\ 
 & =\omega \left( \left[ \begin{matrix}
   O & AB  \\
   CD & O  \\
\end{matrix} \right]+\left[ \begin{matrix}
   O & CD  \\
   AB & O  \\
\end{matrix} \right] \right) \\ 
 & \le \omega \left( \left[ \begin{matrix}
   O & AB  \\
   CD & O  \\
\end{matrix} \right] \right)+\frac{1}{2}\sqrt{4\underset{\theta \in \mathbb{R}}{\mathop{\sup }}\,\left\| \mathfrak R\left[ \begin{matrix}
   O & {{e}^{\textup i\theta }}AB  \\
   {{e}^{\textup i\theta }}CD & O  \\
\end{matrix} \right]\mathfrak R\left[ \begin{matrix}
   O & {{e}^{\textup i\theta }}CD  \\
   {{e}^{\textup i\theta }}AB & O  \\
\end{matrix} \right] \right\|} \\ 
 & =\omega \left( \left[ \begin{matrix}
   O & AB  \\
   CD & O  \\
\end{matrix} \right] \right)+\frac{1}{2}\sqrt{\underset{\theta \in \mathbb{R}}{\mathop{\sup }}\,\left\| \left[ \begin{matrix}
   {{\left( {{e}^{\textup i\theta }}AB+{{e}^{-\textup i\theta }}{{D}^{*}}{{C}^{*}} \right)}^{2}} & O  \\
   O & {{\left( {{e}^{-\textup i\theta }}{{B}^{*}}{{A}^{*}}+{{e}^{\textup i\theta }}CD \right)}^{2}}  \\
\end{matrix} \right] \right\|} \\ 
 & =\omega \left( \left[ \begin{matrix}
   O & AB  \\
   CD & O  \\
\end{matrix} \right] \right)+\frac{1}{2}\sqrt{\underset{\theta \in \mathbb{R}}{\mathop{\sup }}\,\left\| \left[ \begin{matrix}
   {{\left( {{e}^{\textup i\theta }}AB+{{e}^{-\textup i\theta }}{{D}^{*}}{{C}^{*}} \right)}^{2}} & O  \\
   O & {{\left( {{\left( {{e}^{\textup i\theta }}AB+{{e}^{-\textup i\theta }}{{D}^{*}}{{C}^{*}} \right)}^{*}} \right)}^{2}}  \\
\end{matrix} \right] \right\|} \\ 
 & =\omega \left( \left[ \begin{matrix}
   O & AB  \\
   CD & O  \\
\end{matrix} \right] \right)+\frac{1}{2}\sqrt{\underset{\theta \in \mathbb{R}}{\mathop{\sup }}\,\left\| \left[ \begin{matrix}
   {{\left( {{e}^{\textup i\theta }}AB+{{e}^{-\textup i\theta }}{{D}^{*}}{{C}^{*}} \right)}^{2}} & O  \\
   O & {{\left( {{\left( {{e}^{\textup i\theta }}AB+{{e}^{-\textup i\theta }}{{D}^{*}}{{C}^{*}} \right)}^{2}} \right)}^{*}}  \\
\end{matrix} \right] \right\|} \\ 
 & =\omega \left( \left[ \begin{matrix}
   O & AB  \\
   CD & O  \\
\end{matrix} \right] \right)+\frac{1}{2}\sqrt{\underset{\theta \in \mathbb{R}}{\mathop{\sup }}\,\left\| {{\left( {{e}^{\textup i\theta }}AB+{{e}^{-\textup i\theta }}{{D}^{*}}{{C}^{*}} \right)}^{2}} \right\|}. 
\end{aligned}\]
}
Thus,
\[\omega \left( AB+CD \right)\le \omega \left( \left[ \begin{matrix}
   O & AB  \\
   CD & O  \\
\end{matrix} \right] \right)+\frac{1}{2}\sqrt{\underset{\theta \in \mathbb{R}}{\mathop{\sup }}\,\left\| {{\left( {{e}^{\textup i\theta }}AB+{{e}^{-\textup i\theta }}{{D}^{*}}{{C}^{*}} \right)}^{2}} \right\|}.\]
On the other hand,
\[\begin{aligned}
   \left\| {{\left( {{e}^{i\theta }}AB+{{e}^{-i\theta }}{{D}^{*}}{{C}^{*}} \right)}^{2}} \right\|&=\left\| {{e}^{2\textup i\theta }}{{\left( AB \right)}^{2}}+{{e}^{-2\textup i\theta }}{{\left( {{D}^{*}}{{C}^{*}} \right)}^{2}}+AB{{D}^{*}}{{C}^{*}}+{{D}^{*}}{{C}^{*}}AB \right\| \\ 
 & \le \left\|  {{\left( AB \right)}^{2}}\right\|+  \left\|{{\left( {{D}^{*}}{{C}^{*}} \right)}^{2}} \right\|+\left\| AB{{D}^{*}}{{C}^{*}}+{{D}^{*}}{{C}^{*}}AB \right\|  
\end{aligned}\]
where we have used triangle inequality for the usual operator norm to obtain the last line. This completes the proof.
\end{proof}

\begin{corollary}
Let $A,C\in \mathcal B\left( \mathcal H \right)$. Then
\begin{equation}\label{17}
\omega \left( A\pm C \right)\le \omega \left( A\mp C \right)+\sqrt{\left\|  \left\|{{A}^{2}}\right\|+ \left\|{{C}^{2}}\right\|+\left\| A{{C}^{*}}+{{C}^{*}}A \right\| \right\|}.
\end{equation}
In particular,
\begin{equation}\label{19}
{{\left|\; \left\| \mathfrak RA \right\|-\left\| \mathfrak IA \right\| \;\right|}^{2}}\le   \left\|{{A}^{2}}\right\|.
\end{equation}
\end{corollary}
\begin{proof}
From \cite[Theorem 2.4]{2}, we know that
\[\omega \left( \left[ \begin{matrix}
   O & A  \\
   C & O  \\
\end{matrix} \right] \right)\le \frac{1}{2}\left( \omega \left( A+C \right)+\omega \left( A-C \right) \right).\]
This, together with Theorem \ref{8} implies the inequality \eqref{17}. The second inequality can be obtained from  \eqref{17} by setting $C={{A}^{*}}$.
\end{proof}

As a byproduct of inequality \eqref{19}, we have the following result.
\begin{corollary}
Let $A\in \mathcal B\left( \mathcal H \right)$. If ${{A}^{2}}=O$, then $\left\| \mathfrak RA \right\|=\left\| \mathfrak IA \right\|$.
\end{corollary}

We continue with the following general inequality for the product of operators, which implies several consequences for more uncomplicated cases.
\begin{theorem}\label{1}
Let $A,B,C,D,E,F\in \mathcal B\left( \mathcal H \right)$. Then for any $0\le t\le 1$,
\[\begin{aligned}
   {{\omega }^{2}}\left( ABC+DEF \right)&\le \frac{1}{2}\omega \left( \left( {{A}^{*}}{{\left| {{B}^{*}} \right|}^{2\left( 1-t \right)}}A+{{D}^{*}}{{\left| {{E}^{*}} \right|}^{2\left( 1-t \right)}}D \right)\left( {{C}^{*}}{{\left| B \right|}^{2t}}C+{{F}^{*}}{{\left| E \right|}^{2t}}F \right) \right) \\ 
 &\quad +\frac{1}{2}\left\| {{C}^{*}}{{\left| B \right|}^{2t}}C+{{F}^{*}}{{\left| E \right|}^{2t}}F \right\|\left\| A{{\left| {{B}^{*}} \right|}^{2\left( 1-t \right)}}{{A}^{*}}+D{{\left| {{E}^{*}} \right|}^{2\left( 1-t \right)}}{{D}^{*}} \right\|.
\end{aligned}\]
In particular,
\[{{\omega }^{2}}\left( ABC+DEF \right)\le \frac{1}{2}\min \left\{ \alpha ,\beta  \right\},\]
where
\[\alpha =\omega \left( \left( {{\left| A \right|}^{2}}+{{\left| D \right|}^{2}} \right)\left( {{C}^{*}}{{\left| B \right|}^{2}}C+{{F}^{*}}{{\left| E \right|}^{2}}F \right) \right)+\left\| {{C}^{*}}{{\left| B \right|}^{2}}C+{{F}^{*}}{{\left| E \right|}^{2}}F \right\|\left\| {{\left| {{A}^{*}} \right|}^{2}}+{{\left| {{D}^{*}} \right|}^{2}} \right\|,\]
and
\[\beta =\omega \left( \left( {{A}^{*}}{{\left| {{B}^{*}} \right|}^{2}}A+{{D}^{*}}{{\left| {{E}^{*}} \right|}^{2}}D \right)\left( {{\left| C \right|}^{2}}+{{\left| F \right|}^{2}} \right) \right)+\left\| {{\left| C \right|}^{2}}+{{\left| F \right|}^{2}} \right\|\left\| A{{\left| {{B}^{*}} \right|}^{2}}{{A}^{*}}+D{{\left| {{E}^{*}} \right|}^{2}}{{D}^{*}} \right\|.\]
\end{theorem}
\begin{proof}
Let $T=\left[ \begin{matrix}
   ABC+DEF & O  \\
   O & O  \\
\end{matrix} \right]\in\mathcal{B}(\mathcal{H}\oplus\mathcal{H})$, and let $\mathbf x=\left[ \begin{matrix}
   {{x}_{1}}  \\
   {{x}_{2}}  \\
\end{matrix} \right]$ be a unit vector in $\mathcal{H}\oplus\mathcal{H}$. Indeed, ${{\left\| {{x}_{1}} \right\|}^{2}}+{{\left\| {{x}_{2}} \right\|}^{2}}=1$. Then 
{\small	
\[\begin{aligned}
  & {{\left| \left\langle \left[ \begin{matrix}
   ABC+DEF & O  \\
   O & O  \\
\end{matrix} \right]\left[ \begin{matrix}
   {{x}_{1}}  \\
   {{x}_{2}}  \\
\end{matrix} \right],\left[ \begin{matrix}
   {{x}_{1}}  \\
   {{x}_{2}}  \\
\end{matrix} \right] \right\rangle  \right|}^{2}} \\ 
 & ={{\left| \left\langle \left[ \begin{matrix}
   A & D  \\
   O & O  \\
\end{matrix} \right]\left[ \begin{matrix}
   B & O  \\
   O & E  \\
\end{matrix} \right]\left[ \begin{matrix}
   C & O  \\
   F & O  \\
\end{matrix} \right]\left[ \begin{matrix}
   {{x}_{1}}  \\
   {{x}_{2}}  \\
\end{matrix} \right],\left[ \begin{matrix}
   {{x}_{1}}  \\
   {{x}_{2}}  \\
\end{matrix} \right] \right\rangle  \right|}^{2}} \\ 
 & ={{\left| \left\langle \left[ \begin{matrix}
   B & O  \\
   O & E  \\
\end{matrix} \right]\left[ \begin{matrix}
   C & O  \\
   F & O  \\
\end{matrix} \right]\left[ \begin{matrix}
   {{x}_{1}}  \\
   {{x}_{2}}  \\
\end{matrix} \right],\left[ \begin{matrix}
   {{A}^{*}} & O  \\
   {{D}^{*}} & O  \\
\end{matrix} \right]\left[ \begin{matrix}
   {{x}_{1}}  \\
   {{x}_{2}}  \\
\end{matrix} \right] \right\rangle  \right|}^{2}} \\ 
 & \le \left\langle \left[ \begin{matrix}
   {{C}^{*}}{{\left| B \right|}^{2t}}C+{{F}^{*}}{{\left| E \right|}^{2t}}F & O  \\
   O & O  \\
\end{matrix} \right]\left[ \begin{matrix}
   {{x}_{1}}  \\
   {{x}_{2}}  \\
\end{matrix} \right],\left[ \begin{matrix}
   {{x}_{1}}  \\
   {{x}_{2}}  \\
\end{matrix} \right] \right\rangle \left\langle \left[ \begin{matrix}
   A{{\left| {{B}^{*}} \right|}^{2\left( 1-t \right)}}{{A}^{*}}+D{{\left| {{E}^{*}} \right|}^{2\left( 1-t \right)}}{{D}^{*}} & O  \\
   O & O  \\
\end{matrix} \right]\left[ \begin{matrix}
   {{x}_{1}}  \\
   {{x}_{2}}  \\
\end{matrix} \right],\left[ \begin{matrix}
   {{x}_{1}}  \\
   {{x}_{2}}  \\
\end{matrix} \right] \right\rangle  \\ 
&\qquad \text{(by Lemma \ref{14})}\\
 & \le \frac{1}{2}\left| \left\langle \left[ \begin{matrix}
   \left( {{A}^{*}}{{\left| {{B}^{*}} \right|}^{2\left( 1-t \right)}}A+{{D}^{*}}{{\left| {{E}^{*}} \right|}^{2\left( 1-t \right)}}D \right)\left( {{C}^{*}}{{\left| B \right|}^{2t}}C+{{F}^{*}}{{\left| E \right|}^{2t}}F \right) & O  \\
   O & O  \\
\end{matrix} \right]\left[ \begin{matrix}
   {{x}_{1}}  \\
   {{x}_{2}}  \\
\end{matrix} \right],\left[ \begin{matrix}
   {{x}_{1}}  \\
   {{x}_{2}}  \\
\end{matrix} \right] \right\rangle  \right| \\ 
 &\quad +\frac{1}{2}\left\| \left[ \begin{matrix}
   {{C}^{*}}{{\left| B \right|}^{2t}}C+{{F}^{*}}{{\left| E \right|}^{2t}}F & O  \\
   O & O  \\
\end{matrix} \right]\left[ \begin{matrix}
   {{x}_{1}}  \\
   {{x}_{2}}  \\
\end{matrix} \right] \right\|\left\| \left[ \begin{matrix}
   A{{\left| {{B}^{*}} \right|}^{2\left( 1-t \right)}}{{A}^{*}}+D{{\left| {{E}^{*}} \right|}^{2\left( 1-t \right)}}{{D}^{*}} & O  \\
   O & O  \\
\end{matrix} \right]\left[ \begin{matrix}
   {{x}_{1}}  \\
   {{x}_{2}}  \\
\end{matrix} \right] \right\| \\
&\qquad \text{(by Lemma \ref{7})}\\ 
 & \le \frac{1}{2}\omega \left( \left( {{A}^{*}}{{\left| {{B}^{*}} \right|}^{2\left( 1-t \right)}}A+{{D}^{*}}{{\left| {{E}^{*}} \right|}^{2\left( 1-t \right)}}D \right)\left( {{C}^{*}}{{\left| B \right|}^{2t}}C+{{F}^{*}}{{\left| E \right|}^{2t}}F \right) \right) \\ 
 &\quad +\frac{1}{2}\left\| {{C}^{*}}{{\left| B \right|}^{2t}}C+{{F}^{*}}{{\left| E \right|}^{2t}}F \right\|\left\| A{{\left| {{B}^{*}} \right|}^{2\left( 1-t \right)}}{{A}^{*}}+D{{\left| {{E}^{*}} \right|}^{2\left( 1-t \right)}}{{D}^{*}} \right\|.
\end{aligned}\]
}
Therefore,
\[\begin{aligned}
   {{\omega }^{2}}\left( ABC+DEF \right)&\le \frac{1}{2}\omega \left( \left( {{A}^{*}}{{\left| {{B}^{*}} \right|}^{2\left( 1-t \right)}}A+{{D}^{*}}{{\left| {{E}^{*}} \right|}^{2\left( 1-t \right)}}D \right)\left( {{C}^{*}}{{\left| B \right|}^{2t}}C+{{F}^{*}}{{\left| E \right|}^{2t}}F \right) \right) \\ 
 &\quad +\frac{1}{2}\left\| {{C}^{*}}{{\left| B \right|}^{2t}}C+{{F}^{*}}{{\left| E \right|}^{2t}}F \right\|\left\| A{{\left| {{B}^{*}} \right|}^{2\left( 1-t \right)}}{{A}^{*}}+D{{\left| {{E}^{*}} \right|}^{2\left( 1-t \right)}}{{D}^{*}} \right\|,  
\end{aligned}\]
as desired.
\end{proof}
In the following remark and corollaries, we attempt to present some simple consequences from Theorem \ref{1}.
\begin{remark}
\hfill \break
\textup{(I)} The case $B=E=I$, in Theorem \ref{1}, reduces to
\[{{\omega }}\left( AC+DF \right)^{2}\le \frac{1}{2}\left( \omega \left( \left( {{\left| A \right|}^{2}}+{{\left| D \right|}^{2}} \right)\left( {{\left| C \right|}^{2}}+{{\left| F \right|}^{2}} \right) \right)+\left\| {{\left| C \right|}^{2}}+{{\left| F \right|}^{2}} \right\|\left\| {{\left| {{A}^{*}} \right|}^{2}}+{{\left| {{D}^{*}} \right|}^{2}} \right\| \right).\]
\textup{(II)} The case $D=E=F=O$, in Theorem \ref{1}, reduces to
\[{{\omega }}\left( ABC \right)^{2}\le \frac{1}{2}\min \left\{ {{\alpha'}},{{\beta' }} \right\}\]
where
	\[{{\alpha' }}=\omega \left( {{\left| A \right|}^{2}}\left( {{C}^{*}}{{\left| B \right|}^{2}}C \right) \right)+{{\left\| A \right\|}^{2}}\left\| {{C}^{*}}{{\left| B \right|}^{2}}C \right\|,\]
and
	\[{{\beta' }}=\omega \left( \left( {{A}^{*}}{{\left| {{B}^{*}} \right|}^{2}}A \right){{\left| C \right|}^{2}} \right)+{{\left\| C \right\|}^{2}}\left\| A{{\left| {{B}^{*}} \right|}^{2}}{{A}^{*}} \right\|.\]
\end{remark}

\begin{corollary}\label{2}
Let $A,B\in \mathcal B\left( \mathcal H \right)$. Then
\[\omega \left( AB \right)\le \frac{1}{2}\min \left\{ {{\alpha }_{0}},{{\beta }_{0}} \right\},\]
where
\[{{\alpha }_{0}}=\sqrt{2}{{\left\| {{B}^{*}}{{\left| A \right|}^{2}}B+A{{\left| {{B}^{*}} \right|}^{2}}{{A}^{*}} \right\|}^{\frac{1}{2}}},\]
and
\[{{\beta }_{0}}=\left\| {{\left| {{A}^{*}} \right|}^{2}}+{{\left| B \right|}^{2}} \right\|.\]
\end{corollary}
\begin{proof}
Letting $A={{e}^{\textup i\theta }}$, $B=A$, $C=B$, $D={{e}^{-\textup i\theta }}$, $E={{B}^{*}}$, and $F={{A}^{*}}$ in Theorem \ref{1}, and then taking supremum over $\theta \in \mathbb{R}$, we get the desired result.
\end{proof}
The above corollary implies two inequalities, as follows:
$$\omega^2(AB)\leq \frac{1}{2}{{\left\| {{B}^{*}}{{\left| A \right|}^{2}}B+A{{\left| {{B}^{*}} \right|}^{2}}{{A}^{*}} \right\|}}\;{\text{and}}\;\omega(AB)\leq \frac{1}{2}\left\| {{\left| {{A}^{*}} \right|}^{2}}+{{\left| B \right|}^{2}} \right\|.$$
Notice that if $B=I$ in the first inequality, we get the first inequality in Lemma \ref{lem_kitt}. 

\begin{corollary}\label{wAB}
Let $T\in \mathcal B\left( \mathcal H \right)$. Then for any $0\le t\le 1$
\[\omega \left( T \right)\le \frac{1}{2}\min \left\{ {{\alpha }_{1}},{{\beta }_{1}} \right\},\]
where
\[{{\alpha }_{1}}=\sqrt{2}{{\left\| {{\left| T \right|}^{2}}+{{\left| {{T}^{*}} \right|}^{2}} \right\|}^{\frac{1}{2}}},\]
and
\[{{\beta }_{1}}=\left\| {{\left| {{T}^{*}} \right|}^{2\left( 1-t \right)}}+{{\left| T \right|}^{2t}} \right\|.\]
\end{corollary}
We notice that the above corollary implies two particular inequalities as follows:
$$\omega^2(T)\leq \frac{1}{2}\|\;|T|^2+|T^*|^2\|\;{\text{and}}\;\omega(T)\leq \frac{1}{2}\|\;|T|+|T^*|\;\|,$$ 
which are the two inequalities from Lemma \ref{lem_kitt}. This shows how Theorem \ref{1} can be used to obtain some inequalities from the literature.

We also have the following result for the product of two operators. This result will be used to obtain a main result about the Aluthge transform.

\begin{theorem}\label{18}
	Let $A,B\in \mathcal B\left( \mathcal H \right)$. Then
{\scriptsize	
\[\begin{aligned}
  & \omega \left( AB \right) \\ 
 & \le \frac{1}{2}\sqrt{\frac{1}{4}{{\left\| \frac{\left\| B \right\|}{\left\| A \right\|}{{\left| A \right|}^{2}}+\frac{\left\| A \right\|}{\left\| B \right\|}{{\left| {{B}^{*}} \right|}^{2}} \right\|}^{2}}+{{\omega }}\left( BA \right)^{2}+\frac{1}{2}\omega \left( \left( \frac{\left\| B \right\|}{\left\| A \right\|}{{\left| A \right|}^{2}}+\frac{\left\| A \right\|}{\left\| B \right\|}{{\left| {{B}^{*}} \right|}^{2}} \right)BA+BA\left( \frac{\left\| B \right\|}{\left\| A \right\|}{{\left| A \right|}^{2}}+\frac{\left\| A \right\|}{\left\| B \right\|}{{\left| {{B}^{*}} \right|}^{2}} \right) \right)}. 
\end{aligned}\]
}
\end{theorem}
\begin{proof}
By the Polarization identity, we can write, for any unit vector $x\in \mathcal H$,
{\small
\[\begin{aligned}
  & \left\langle \mathfrak R\left( {{e}^{\textup i\theta }}AB \right)x,x \right\rangle  \\ 
 & =\mathfrak R\left\langle {{e}^{\textup i\theta }}ABx,x \right\rangle  \\ 
 & =\mathfrak R\left\langle {{e}^{\textup i\theta }}Bx,{{A}^{*}}x \right\rangle  \\ 
 & =\frac{1}{4}{{\left\| \left( {{e}^{\textup i\theta }}B+{{A}^{*}} \right)x \right\|}^{2}}-\frac{1}{4}{{\left\| \left( {{e}^{\textup i\theta }}B-{{A}^{*}} \right)x \right\|}^{2}} \\ 
 & \le \frac{1}{4}{{\left\| \left( {{e}^{\textup i\theta }}B+{{A}^{*}} \right)x \right\|}^{2}} \\ 
 & \le \frac{1}{4}{{\left\| {{e}^{\textup i\theta }}B+{{A}^{*}} \right\|}^{2}} \\ 
 & =\frac{1}{4}\left\| \left( {{e}^{\textup i\theta }}B+{{A}^{*}} \right){{\left( {{e}^{\textup i\theta }}B+{{A}^{*}} \right)}^{*}} \right\| \\ 
 & =\frac{1}{2}\left\| \frac{{{\left| A \right|}^{2}}+{{\left| {{B}^{*}} \right|}^{2}}}{2}+{{\operatorname{\mathfrak Re}}^{\textup i\theta }}BA \right\| \\ 
 & =\frac{1}{2}{{\left\| {{\left( \frac{{{\left| A \right|}^{2}}+{{\left| {{B}^{*}} \right|}^{2}}}{2} \right)}^{2}}+{{\left( {{\operatorname{\mathfrak Re}}^{\textup i\theta }}BA \right)}^{2}}+\left( \frac{{{\left| A \right|}^{2}}+{{\left| {{B}^{*}} \right|}^{2}}}{2} \right){{\operatorname{\mathfrak Re}}^{\textup i\theta }}BA+{{\operatorname{\mathfrak Re}}^{\textup i\theta }}BA\left( \frac{{{\left| A \right|}^{2}}+{{\left| {{B}^{*}} \right|}^{2}}}{2} \right) \right\|}^{\frac{1}{2}}} \\ 
 & =\frac{1}{2}{{\left\| {{\left( \frac{{{\left| A \right|}^{2}}+{{\left| {{B}^{*}} \right|}^{2}}}{2} \right)}^{2}}+{{\left( {{\operatorname{\mathfrak Re}}^{\textup i\theta }}BA \right)}^{2}}+\frac{1}{2}\left( \mathfrak R{{e}^{\textup i\theta }}\left( \left( {{\left| A \right|}^{2}}+{{\left| {{B}^{*}} \right|}^{2}} \right)BA+BA\left( {{\left| A \right|}^{2}}+{{\left| {{B}^{*}} \right|}^{2}} \right) \right) \right) \right\|}^{\frac{1}{2}}} \\ 
 & =\frac{1}{2}\sqrt{\frac{1}{4}{{\left\| {{\left| A \right|}^{2}}+{{\left| {{B}^{*}} \right|}^{2}} \right\|}^{2}}+{{\left\| {{\operatorname{\mathfrak Re}}^{\textup i\theta }}BA \right\|}^{2}}+\frac{1}{2}\left\| \mathfrak R{{e}^{\textup i\theta }}\left( \left( {{\left| A \right|}^{2}}+{{\left| {{B}^{*}} \right|}^{2}} \right)BA+BA\left( {{\left| A \right|}^{2}}+{{\left| {{B}^{*}} \right|}^{2}} \right) \right) \right\|} \\ 
 & \le \frac{1}{2}\sqrt{\frac{1}{4}{{\left\| {{\left| A \right|}^{2}}+{{\left| {{B}^{*}} \right|}^{2}} \right\|}^{2}}+{{\omega }}\left( BA \right)^{2}+\frac{1}{2}\omega \left( \left( {{\left| A \right|}^{2}}+{{\left| {{B}^{*}} \right|}^{2}} \right)BA+BA\left( {{\left| A \right|}^{2}}+{{\left| {{B}^{*}} \right|}^{2}} \right) \right)}.
\end{aligned}\]
}
Thus,
{\small
\begin{equation}\label{15}
\omega \left( AB \right)\le \frac{1}{2}\sqrt{\frac{1}{4}{{\left\| {{\left| A \right|}^{2}}+{{\left| {{B}^{*}} \right|}^{2}} \right\|}^{2}}+{{\omega }}\left( BA \right)^{2}+\frac{1}{2}\omega \left( \left( {{\left| A \right|}^{2}}+{{\left| {{B}^{*}} \right|}^{2}} \right)BA+BA\left( {{\left| A \right|}^{2}}+{{\left| {{B}^{*}} \right|}^{2}} \right) \right)}.
\end{equation}}
Replacing $A$ by $\sqrt{\frac{\left\| B \right\|}{\left\| A \right\|}}A$ and $B$ by $\sqrt{\frac{\left\| A \right\|}{\left\| B \right\|}}B$ in \eqref{15}, we get
{\scriptsize	
\[\begin{aligned}
  & \omega \left( AB \right) \\ 
 & \le \frac{1}{2}\sqrt{\frac{1}{4}{{\left\| \frac{\left\| B \right\|}{\left\| A \right\|}{{\left| A \right|}^{2}}+\frac{\left\| A \right\|}{\left\| B \right\|}{{\left| {{B}^{*}} \right|}^{2}} \right\|}^{2}}+{{\omega }}\left( BA \right)^{2}+\frac{1}{2}\omega \left( \left( \frac{\left\| B \right\|}{\left\| A \right\|}{{\left| A \right|}^{2}}+\frac{\left\| A \right\|}{\left\| B \right\|}{{\left| {{B}^{*}} \right|}^{2}} \right)BA+BA\left( \frac{\left\| B \right\|}{\left\| A \right\|}{{\left| A \right|}^{2}}+\frac{\left\| A \right\|}{\left\| B \right\|}{{\left| {{B}^{*}} \right|}^{2}} \right) \right)},
\end{aligned}\]
}
as desired.
\end{proof}

It is worth mentioning here that, Theorem \ref{18} improves \cite[Theorem 3.2]{5} when $X=I$.

Now we discuss inequalities that involve the Aluthge transform.
The following numerical radius inequality in terms of the Aluthge transform improves the upper bound obtained in \cite[Theorem 3.2]{3}.

\begin{theorem}\label{9}
Let $T\in \mathcal B\left( \mathcal H \right)$ and $T=U\left| T \right|$ be the polar decomposition of $T$, and let ${{\widetilde{T}}_{t}}={{\left| T \right|}^{t}}U{{\left| T \right|}^{1-t}}\left( 0\le t\le 1 \right)$ be the generalized Aluthge transformation of $T$. Then
{\footnotesize		
\[\begin{aligned}
  & \omega \left( T \right) \\ 
 & \le \frac{1}{2}\sqrt{{{\left\| T \right\|}^{2}}+{{\omega }}\left( \widetilde{{{T}_{t}}} \right)^{2}+\frac{1}{2}\omega \left( \left( {{\left\| T \right\|}^{2t-1}}{{\left| T \right|}^{2\left( 1-t \right)}}+{{\left\| T \right\|}^{1-2t}}{{\left| T \right|}^{2t}} \right)\widetilde{{{T}_{t}}}+\widetilde{{{T}_{t}}}\left( {{\left\| T \right\|}^{2t-1}}{{\left| T \right|}^{2\left( 1-t \right)}}+{{\left\| T \right\|}^{1-2t}}{{\left| T \right|}^{2t}} \right) \right)}. 
\end{aligned}\]
}
\end{theorem}
\begin{proof}
Replacing $A=U{{\left| T \right|}^{1-t}}$ and $B={{\left| T \right|}^{t}}$, in Theorem \ref{18}, and noting ${{\left| {{B}^{*}} \right|}^{2}}={{\left| B \right|}^{2}}={{\left| T \right|}^{2t}}$, ${{\left| {{A}^{*}} \right|}^{2}}=U{{\left| T \right|}^{1-t}}{{\left| T \right|}^{1-t}}{{U}^{*}}=U{{\left| T \right|}^{2\left( 1-t \right)}}{{U}^{*}}={{\left| {{T}^{*}} \right|}^{2\left( 1-t \right)}}$ (see \cite[(1)]{8}), ${{\left| A \right|}^{2}}={{\left| T \right|}^{1-t}}{{U}^{*}}U{{\left| T \right|}^{1-t}}={{\left| T \right|}^{2\left( 1-t \right)}}$, $AB=U\left| T \right|=T$, $BA={{\left| T \right|}^{t}}U{{\left| T \right|}^{1-t}}={{\widetilde{T}}_{t}}$. \\
Here we also used the following identity
\begin{equation}\label{4}
\left\| {{\left\| T \right\|}^{2t-1}}{{\left| T \right|}^{2\left( 1-t \right)}}+{{\left\| T \right\|}^{1-2t}}{{\left| T \right|}^{2t}} \right\|= 2\left\| T \right\|.
\end{equation}
In fact,
\begin{equation}\label{2T}
\begin{aligned}
	2\|T\|&=2\|U\left| T \right|\|\\
	&=2\left\|U{{\left| T \right|}^{1-t}}{{\left| T \right|}^{t}}\right\|\\
	&=2\left\|\left(\frac{{{\left\| T \right\|}^{t}}}{{{\left\| T \right\|}^{1-t}}}\right)^{\frac 12} U{{\left| T \right|}^{1-t}}\left(\frac{{{\left\| T \right\|}^{1-t}}}{{{\left\| T \right\|}^{t}}}\right)^{\frac 12}{{\left| T \right|}^{t}}\right\|.
\end{aligned}
\end{equation}
From \cite{Mc} we also know that
\begin{equation}\label{3}
2\left\| AXB\right\| \leq \left\| A^*AX+XBB^{*}\right\|,
\end{equation}
where $A,B,X\in \mathcal{B}(\mathcal{H})$. Repalcing $A=\left(\frac{{{\left\| T \right\|}^{t}}}{{{\left\| T \right\|}^{1-t}}}\right)^{\frac 12} U{{\left| T \right|}^{1-t}}$, $X=I$, and $B=\left(\frac{{{\left\| T \right\|}^{1-t}}}{{{\left\| T \right\|}^{t}}}\right)^{\frac 12}{{\left| T \right|}^{t}}$ in  \eqref{3}, we get 
\[\begin{aligned}\label{2T2}
2\|T\|&=2\left\|\left(\frac{{{\left\| T \right\|}^{t}}}{{{\left\| T \right\|}^{1-t}}}\right)^{\frac 12} U{{\left| T \right|}^{1-t}}\left(\frac{{{\left\| T \right\|}^{1-t}}}{{{\left\| T \right\|}^{t}}}\right)^{\frac 12}{{\left| T \right|}^{t}}\right\|\quad \text{\text{(by \eqref{2T})}} \\
&\le\left\|\left\| T \right\|^{2t-1}|T|^{2(1-t)}+\|T\|^{1-2t}{{\left| T \right|}^{2t}}\right\|.
\end{aligned}\]
On the other hand, by employing the triangle inequality for the usual operator norm, we have
\begin{equation}\label{5}
\left\| {{\left\| T \right\|}^{2t-1}}{{\left| T \right|}^{2\left( 1-t \right)}}+{{\left\| T \right\|}^{1-2t}}{{\left| T \right|}^{2t}} \right\|\le 2\left\| T \right\|.
\end{equation}
Combining inequalities \eqref{2T2} and \eqref{5}, we conclude \eqref{4}.
\end{proof}

\begin{remark}
By \eqref{12}, we infer that
{\footnotesize	
\[\begin{aligned}
  & \omega \left( T \right) \\ 
 & \le \frac{1}{2}\sqrt{{{\left\| T \right\|}^{2}}+{{\omega }^{2}}\left( \widetilde{{{T}_{t}}} \right)+\frac{1}{2}\omega \left( \left( {{\left\| T \right\|}^{2t-1}}{{\left| T \right|}^{2\left( 1-t \right)}}+{{\left\| T \right\|}^{1-2t}}{{\left| T \right|}^{2t}} \right)\widetilde{{{T}_{t}}}+\widetilde{{{T}_{t}}}\left( {{\left\| T \right\|}^{2t-1}}{{\left| T \right|}^{2\left( 1-t \right)}}+{{\left\| T \right\|}^{1-2t}}{{\left| T \right|}^{2t}} \right) \right)} \\ 
 & \le \frac{1}{2}\sqrt{{{\left\| T \right\|}^{2}}+{{\omega }^{2}}\left( \widetilde{{{T}_{t}}} \right)+\omega \left( \widetilde{{{T}_{t}}} \right)\left\| {{\left\| T \right\|}^{2t-1}}{{\left| T \right|}^{2\left( 1-t \right)}}+{{\left\| T \right\|}^{1-2t}}{{\left| T \right|}^{2t}} \right\|} \\ 
 & =\frac{1}{2}\sqrt{{{\left\| T \right\|}^{2}}+{{\omega }^{2}}\left( \widetilde{{{T}_{t}}} \right)+2\omega \left( \widetilde{{{T}_{t}}} \right)\left\| T \right\|} \\ 
 & =\frac{1}{2}\left( \left\| T \right\|+\omega \left( \widetilde{{{T}_{t}}} \right) \right). 
\end{aligned}\]
}
This shows that Theorem \ref{9} is an improvement of  the inequality
\[\omega \left( T \right)\le \frac{1}{2}\left( \omega \left( {{\widetilde{T}}_{t}} \right)+\left\| T \right\| \right)\]
obtained by Abu Omar and Kittaneh \cite[Theorem 3.2]{3}. Of course, the case $t={1}/{2}\;$ also improves to the main result of \cite{4}.
\end{remark}

\begin{remark}
It follows from the inequality \eqref{15} that
\[\begin{aligned}
  & \omega \left( T \right) \\ 
 & \le \frac{1}{2}\sqrt{\frac{1}{4}{{\left\| {{\left| T \right|}^{2\left( 1-t \right)}}+{{\left| T \right|}^{2t}} \right\|}^{2}}+{{\omega }^{2}}\left( \widetilde{{{T}_{t}}} \right)+\frac{1}{2}\omega \left( \left( {{\left| T \right|}^{2\left( 1-t \right)}}+{{\left| T \right|}^{2t}} \right)\widetilde{{{T}_{t}}}+\widetilde{{{T}_{t}}}\left( {{\left| T \right|}^{2\left( 1-t \right)}}+{{\left| T \right|}^{2t}} \right) \right)} \\ 
 & \le \frac{1}{2}\sqrt{\frac{1}{4}{{\left\| {{\left| T \right|}^{2\left( 1-t \right)}}+{{\left| T \right|}^{2t}} \right\|}^{2}}+{{\omega }^{2}}\left( \widetilde{{{T}_{t}}} \right)+\omega \left( \widetilde{{{T}_{t}}} \right)\left\| {{\left| T \right|}^{2\left( 1-t \right)}}+{{\left| T \right|}^{2t}} \right\|} \\ 
 & =\frac{1}{4}\left\| {{\left| T \right|}^{2\left( 1-t \right)}}+{{\left| T \right|}^{2t}} \right\|+\frac{1}{2}\omega \left( \widetilde{{{T}_{t}}} \right). 
\end{aligned}\]
This indicates that our result improves \cite[Corollary 2.11]{BBP}.
\end{remark}

\begin{remark}
Combining the first inequality in \eqref{eq_equivalent} and Theorem \ref{9}, we conclude that  if $\widetilde{{{T}_{t}}}=0$, then 	${1}/{2}\;\left\| T \right\|=\omega \left( T \right)$. 
\end{remark}

\vskip 0.3 true cm 

\noindent{\tiny (M. Sababheh) Vice president, Princess Sumaya University for Technology, Amman, Jordan}
	
\noindent	{\tiny\textit{E-mail address:} sababheh@psut.edu.jo; sababheh@yahoo.com}

\vskip 0.3 true cm

\noindent{\tiny (C. Conde)  Instituto de Ciencias, Universidad Nacional de General Sarmiento  and  Consejo Nacional de Investigaciones Cient\'ificas y Tecnicas, Argentina}

\noindent{\tiny \textit{E-mail address:} cconde@campus.ungs.edu.ar}

\vskip 0.3 true cm

\noindent{\tiny (H. R. Moradi) Department of Mathematics, Payame Noor University (PNU), P.O. Box, 19395-4697, Tehran, Iran
	
\noindent	\textit{E-mail address:} hrmoradi@mshdiau.ac.ir}

%-----------------------------------------------------------------------------
%-----------------------------------------------------------------------------
\end{document}